\documentclass{siamltex1213}
\usepackage{amsfonts,amsmath,amssymb}

\usepackage{MnSymbol}
\makeatletter
\newdimen\rh@wd
\newdimen\rh@hta
\newdimen\rh@htb
\newbox\rh@box
\def\rh@measure#1{\setbox\rh@box=\hbox{$#1$}\rh@wd=\wd\rh@box \rh@hta=\ht\rh@box}

\def\widecheck#1{\rh@measure{#1}%
  \setbox\rh@box=\hbox{$\widehat{\vrule height \rh@hta width\z@ \kern\rh@wd}$}%
  \rh@htb=\ht\rh@box \advance\rh@htb\rh@hta \advance\rh@htb\p@
  \ooalign{$\vrule height \ht\rh@box width\z@ #1$\cr
           \raise\rh@htb\hbox{\scalebox{1}[-1]{\box\rh@box}}\cr}}
\makeatother
\makeatletter
\let\save@mathaccent\mathaccent
\newcommand*\if@single[3]{%
  \setbox0\hbox{${\mathaccent"0362{#1}}^H$}%
  \setbox2\hbox{${\mathaccent"0362{\kern0pt#1}}^H$}%
  \ifdim\ht0=\ht2 #3\else #2\fi
  }
\newcommand*\rel@kern[1]{\kern#1\dimexpr\macc@kerna}
\newcommand*\widebar[1]{\@ifnextchar^{{\wide@bar{#1}{0}}}{\wide@bar{#1}{1}}}
\newcommand*\wide@bar[2]{\if@single{#1}{\wide@bar@{#1}{#2}{1}}{\wide@bar@{#1}{#2}{2}}}
\newcommand*\wide@bar@[3]{%
  \begingroup
  \def\mathaccent##1##2{%
    \let\mathaccent\save@mathaccent
    \if#32 \let\macc@nucleus\first@char \fi
    \setbox\z@\hbox{$\macc@style{\macc@nucleus}_{}$}%
    \setbox\tw@\hbox{$\macc@style{\macc@nucleus}{}_{}$}%
    \dimen@\wd\tw@
    \advance\dimen@-\wd\z@
    \divide\dimen@ 3
    \@tempdima\wd\tw@
    \advance\@tempdima-\scriptspace
    \divide\@tempdima 10
    \advance\dimen@-\@tempdima
    \ifdim\dimen@>\z@ \dimen@0pt\fi
    \rel@kern{0.6}\kern-\dimen@
    \if#31
      \overline{\rel@kern{-0.6}\kern\dimen@\macc@nucleus\rel@kern{0.4}\kern\dimen@}%
      \advance\dimen@0.4\dimexpr\macc@kerna
      \let\final@kern#2%
      \ifdim\dimen@<\z@ \let\final@kern1\fi
      \if\final@kern1 \kern-\dimen@\fi
    \else
      \overline{\rel@kern{-0.6}\kern\dimen@#1}%
    \fi
  }%
  \macc@depth\@ne
  \let\math@bgroup\@empty \let\math@egroup\macc@set@skewchar
  \mathsurround\z@ \frozen@everymath{\mathgroup\macc@group\relax}%
  \macc@set@skewchar\relax
  \let\mathaccentV\macc@nested@a
  \if#31
    \macc@nested@a\relax111{#1}%
  \else
    \def\gobble@till@marker##1\endmarker{}%
    \futurelet\first@char\gobble@till@marker#1\endmarker
    \ifcat\noexpand\first@char A\else
      \def\first@char{}%
    \fi
    \macc@nested@a\relax111{\first@char}%
  \fi
  \endgroup
}
\usepackage{hyperref}

\newcommand{\RR}{{\mathbb{R}}}

\newtheorem{example}[theorem]{Example}

\title{Computationally enhanced projection methods 
for symmetric Sylvester and Lyapunov matrix equations\thanks{Version of February 1, 2017}.
This research is supported in part by the FARB12SIMO grant of the Universit\`a di Bologna,
 and in part by INdAM-GNCS under the 
2015 Project  \emph{Metodi di regolarizzazione per problemi di ottimizzazione e applicazioni.}} 
\author{Davide Palitta\thanks{Dipartimento di Matematica, Universit\`a di Bologna, %
Piazza di Porta S. Donato, 5, I-40127 Bologna, Italy ({\tt davide.palitta3@unibo.it}).}  
\and Valeria Simoncini\thanks{Dipartimento di Matematica, Universit\`a di Bologna, 
Piazza di Porta S. Donato, 5, I-40127 Bologna, Italy ({\tt valeria.simoncini@unibo.it}),
and IMATI-CNR, Pavia, Italy.}}

\begin{document}
\bibliographystyle{siam}
\pagestyle{myheadings}
\markboth{D. Palitta and V.~Simoncini}{Enhanced projection methods
for symmetric Sylvester equations}

\maketitle

\begin{abstract}
In the numerical treatment of large-scale Sylvester and Lyapunov equations,
projection methods require solving a reduced problem to check convergence.
As the approximation space expands, this solution takes an increasing portion of
the overall computational effort.
When data are symmetric, we show that the Frobenius norm of the residual matrix 
can be computed
at significantly lower cost than with available methods, without 
explicitly solving the reduced problem.
For certain classes of problems,
the new residual norm expression combined with a memory-reducing device
make classical Krylov strategies competitive with respect to more recent
projection methods.
Numerical experiments illustrate the effectiveness of the new implementation for
standard and extended Krylov subspace methods.
\end{abstract}
%\keywords{}
%\subclass{}



%%%%%%%%%%%%%%%%%%%%%%%%%%%%%%%%%%%%
\section{Introduction}
Consider the Sylvester matrix equation
\begin{equation}\label{Lyapeq.}
 AX+XB+C_1C_2^T=0,\quad A\in\mathbb{R}^{n_1\times n_1},\,\, B\in\mathbb{R}^{n_2\times n_2},\,\,
 C_1\in\mathbb{R}^{n_1\times s}, C_2\in\mathbb{R}^{n_2\times s}
\end{equation}
where $A, B$ are very large and sparse, symmetric negative definite matrices, while 
$C_1, C_2 \ne 0$ are tall, that is $s\ll n_1, n_2$.
Under these hypotheses, there exists a unique solution matrix $X$. 
%positive semidefinite solution $X$ \cite{Snyders1970} (CHECK REF).
This kind of matrix equation arises in many applications, from the analysis of continuous-time linear dynamical systems
to eigenvalue problems and the discretization of self-adjoint elliptic PDEs; see, e.g., \cite{Antoulas.05},
 and \cite{Simoncini2014} for a recent survey.
Although $A$ and $B$ are sparse, the solution $X$ is in general
dense so that storing it may be unfeasible for large-scale problems.
On the other hand, under certain hypotheses on the spectral distribution of $A$ and $B$,
 the singular values of $X$ present a fast decay, see, e.g., \cite{Penzl2000}, thus
justifying the search for a low-rank approximation $\widetilde X=Z_1Z_2^T$ to $X$ so that only these two
 tall matrices are actually computed and stored.
To simplify the presentation of what follows, from now on we will focus on the case of
the Lyapunov matrix equation, that is $B=A$ ($n\equiv n_1=n_2$) and $C\equiv C_1=C_2$, so that $X$ will be square, symmetric
and positive semidefinite \cite{Snyders1970}. In later sections we will describe how to 
naturally treat the general case with $A$ and $B$ distinct and not necessarily with the
same dimensions, and different $C_1, C_2$.

For the Lyapunov equation,
projection methods compute the numerical solution $\widetilde X$ in a sequence of nested subspaces,
$\mathcal{K}_m\subseteq \mathcal{K}_{m+1} \subseteq \RR^n$, $m\ge 1$. The approximation, usually denoted by $X_m$, is
written as the product of 
low-rank matrices $X_m= V_m Y_m V_m^T$ where $\mathcal{K}_m={\rm Range}(V_m)$ and
the columns of $V_m$ are far fewer than $n$.
The quality and effectiveness of the approximation process depend on how
much spectral information is captured by $\mathcal{K}_m$, without the space dimension being too large. 
The matrix $Y_m$ is determined by solving a related (reduced) problem, whose dimension depends on the approximation
space dimension.  To check convergence, the residual matrix norm is monitored at each iteration
by using $Y_m$ but
without explicitly computing the large and dense residual matrix $R_m = AX_m+X_mA+CC^T$ \cite{Simoncini2014}.
The solution of the reduced problem is meant to account for a low percentage of the
overall computation cost. Unfortunately, this cost grows nonlinearly with the space dimension,
therefore  solving the reduced problem may become very expensive if a large approximation space
is needed.

A classical choice for $\mathcal{K}_m$ is the (standard) block Krylov subspace 
$\mathbf{K}^\square_m(A,C):=\mbox{Range}\{[C,AC,\dots,A^{m-1}C]\}$
\cite{Gutknecht2006}, whose basis can be generated iteratively
by means of the block Lanczos procedure.   Numerical experiments
 show that $\mathbf{K}^\square_m(A,C)$
may need to be quite large before a satisfactory approximate solution is obtained \cite{Penzl2000a},\cite{Simoncini2007}.
This large number of iterations causes high computational and memory demands.
% both
%an increasing cost for computing the residual norm and a often unfeasible memory requirement for storing the basis,
%i.e., the dense matrix $V_m$. 
More recent alternatives include projection onto extended or 
rational Krylov subspaces \cite{Simoncini2007},\cite{Druskin.Simoncini.11}, 
or the use of explicit iterations for the approximate solution \cite{Penzl2000a}; see 
the thorough presentation in \cite{Simoncini2014}. Extended and more generally rational
Krylov subspaces contain richer spectral
information, that allow for a significantly lower subspace dimension,
at the cost of more expensive computations per iteration, since $s$ system solves with
the coefficients matrix are required at each iteration.

%The Standard Krylov subspace method was thus abandoned in the large-scale matrix equation context and
%several numerical strategies were developed in the last decades to solve (\ref{Lyapeq.}),
%see, e.g., the recent survey \cite{Simoncini2014}.
%In 2007, Simoncini introduced in \cite{Simoncini2007} the employment 
%of the Extended Krylov subspace 
%$\mathbf{EK}^\square_m(A,B):=\mbox{Range}\{[B,A^{-1}B,\dots,A^{m-1}B,A^{-m}B]\}$ for solving (\ref{Lyapeq.}).
%Since this approximation space contains much more spectral information than $\mathbf{K}^\square_m(A,B)$, the resulting 
%projection methods typically provides a faster convergence in terms of number of iterations, see, e.g., \cite{Simoncini2007}, 
%in spite of a more expensive basis construction due to linear system solves.

We devise a strategy that significantly  reduces the computational cost of evaluating the residual norm for both 
$\mathbf{K}^\square_m$ and the extended Krylov subspace 
$\mathbf{EK}^\square_m(A,C):=\mbox{Range}\{[C,A^{-1}C,\dots,A^{m-1}C,A^{-m}C]\}$.
In case of $\mathbf{K}^\square_m$ a ``two-pass'' strategy is implemented to avoid
storing the whole basis $V_m$; % but only the last $3s$ vectors of the Lanczos iteration;
see \cite{Kressner2008} for earlier use of this device in the same setting, and, e.g., \cite{Frommer2008}
in the matrix function context.
%This strategy is in general not necessary for $\mathbf{EK}^\square_m$
%as the convergence usually occurs in a very moderate number of iterations.

%MODIFY In Section \ref{SK_twopass}
%we illustrate an implementation of a ``two-pass'' strategy, already presented in \cite{Kressner2008}. This memory-saving
%procedure allows us to store only a portion of the whole basis during the construction of the approximation space,
%in particular, only the $3s$ basis vectors needed in the orthogonalization process are stored.
%The matrix $Z_m=V_m\widehat Y$ is then computed on-the-fly during a second Lanczos pass. 

Throughout the paper, Greek bold letters ($\pmb{\alpha}$) will denote $s\times s$ matrices, 
while roman capital letters ($A$)
larger ones. In particular $E_i\in\mathbb{R}^{sm\times s}$ will denote the $i$th block of $s$ columns 
of the identity matrix $I_{sm}\in\mathbb{R}^{sm\times sm}$.
Scalar quantities will be denoted by Greek letters ($\alpha$).

Here is a synopsis of the paper. In Section~\ref{Projection methods} the basic tools of projection methods for solving 
(\ref{Lyapeq.}) are recalled. In Section \ref{SK_residual} 
we present a cheap residual norm computation whose implementation is discussed in Section~\ref{sec:Algorithm}.
The two-pass strategy for $\mathbf{K}^\square_m(A,C)$ is examined in Section \ref{SK_twopass}.
In Section~\ref{Extended} we extend the residual computation 
to $\mathbf{EK}^\square_m(A,C)$. Section \ref{Sylvester_intro} discusses the generalization
of this procedure to the case of the Sylvester equation in (\ref{Lyapeq.}).
In particular, Section~\ref{The Sylvester case} analyzes the case when both coefficient matrices are large,
while Section \ref{LargeAsmallB} discusses problems where one of them has small dimensions.
Numerical examples illustrating the effectiveness of our strategy are reported in 
Section~\ref{Numerical Experiments},
while our conclusions are given in Section \ref{Conclusions}.

%%%%%%%%%%%%%%%%%%%%%%%%%%%%%%%%%%%%%%%%%%%%%%%%%%%%%%%%%%%%%%%%%%%%%%%%%%%%%%%%%%%%%%%%%%%%%%%%%%%%%%%%%%%%%%%%%%%%

\section{Galerkin projection methods}\label{Projection methods}

Consider a subspace
$\mathcal{K}_m$ spanned by the orthonormal columns of the
 matrix $V_m=[\mathcal{V}_1,\dots,\mathcal{V}_m]\in\mathbb{R}^{n\times sm}$ 
and seek an approximate solution $X_m$ to (\ref{Lyapeq.}) of the form $X_m=V_mY_mV_m^T$ with $Y_m$ symmetric and 
positive semidefinite, and residual matrix $R_m=AX_m+X_mA+CC^T$.
With the matrix inner product
$$
\langle Q,P\rangle_F:=\mbox{trace}(P^TQ), \qquad Q,P\in\mathbb{R}^{n_1\times n_2},
$$
the matrix $Y_m$ can be determined by imposing an orthogonality (Galerkin) condition on the residual with respect to 
this inner product,
\begin{equation}\label{reseq.}
 R_m \perp\mathcal{K}_m \quad \Leftrightarrow\quad V_m^TR_mV_m=0.
\end{equation}
Substituting $R_m$ into (\ref{reseq.}), we obtain
$V_m^TAX_mV_m+V_m^TX_mAV_m+V_m^TCC^TV_m=0$,
that is
\begin{equation}\label{reseqbis}
 \left(V_m^TAV_m\right)Y_mV_m^TV_m+V_m^TV_mY_m\left(V_m^TAV_m\right)+V_m^TCC^TV_m = 0.
\end{equation}
We assume $\mbox{Range}(V_1)=\mbox{Range}(C)$, that is $C=V_1\pmb{\gamma}$ for some 
nonsingular
$\pmb{\gamma}\in\mathbb{R}^{s\times s}$. Since $V_m$ has orthonormal columns, $V_m^TC=E_1\pmb{\gamma}$ and 
equation (\ref{reseqbis}) can be written as
\begin{equation}\label{reducesLyap}
T_mY_m+Y_mT_m+E_1\pmb{\gamma\gamma}^TE_1^T = 0,
\end{equation}
where $T_m:=V_m^TAV_m$ is symmetric and negative definite. 
The orthogonalization procedure employed in building $V_m$ determines the 
sparsity pattern of $T_m$. In particular, for $\mathcal{K}_m=\mathbf{K}^\square_m(A,C)$, the block Lanczos process 
produces a block tridiagonal matrix $T_m$ with blocks of size $s$,
$$ T_m=\left(\begin{array}{lll}       
        \pmb{\tau}_{11} &\quad \pmb{\tau}_{12} &   \\
        \pmb{\tau}_{21} & \quad \pmb{\tau}_{22} & \quad \pmb{\tau}_{23}   \\
         &  \ddots &   \ddots \quad \quad  \ddots   \\
        &   \quad \ddots  &\quad \ddots \qquad \pmb{\tau}_{m-1,m} \\
    & &  \pmb{\tau}_{m,m-1}\quad  \pmb{\tau}_{m,m}  \\
       \end{array}\right).
$$
As long as $m$ is of moderate size, 
methods based on the Schur decomposition of the coefficient matrix $T_m$ can 
be employed to solve equation (\ref{reducesLyap}), see, e.g., \cite{Bartels1972}, 
\cite{Hammarling1982}. 

The last $s$ columns (or rows) of the solution matrix $Y_m$ are employed to compute the 
residual norm. In particular, letting $\underline{T}_m=V_{m+1}^TAV_{m}$, it was shown in \cite{Jaimoukha1994} that
the norm of the residual in (\ref{reseq.}) satisfies
\begin{equation}\label{reseq1.}
 \|R_m\|_F={\sqrt{2}}\|Y_m\underline{T}_m^TE_{m+1}\|_F={\sqrt{2}}\|Y_mE_m\pmb{\tau}^T_{m+1,m}\|_F.
\end{equation}
%%
%These methods are based on decompositions of the coefficient matrix $T_m$ and 
%therefore, .
The matrix $Y_m$ is determined by solving (\ref{reducesLyap}), and it is again symmetric and
positive semidefinite.
At convergence, the backward transformation $X_m=V_mY_mV_m^T$ is never explicitly computed or stored.
Instead, we factorize $Y_m$ as
%dense matrix. As the right-hand side in (\ref{reducesLyap}) is symmetric and positive semidefinite,
%similarly to the exact solution $X$ to (\ref{Lyapeq.}), also $Y_m$ is symmetric and positive semidefinite
%justifying its factorization,
%
\begin{eqnarray}\label{eqn:Y}
Y_m= \widehat Y \widehat Y^T, \quad \widehat Y\in\mathbb{R}^{sm\times sm},
\end{eqnarray}
from which a low-rank factor of $X_m$ is obtained as $Z_m=V_m\widehat Y\in\mathbb{R}^{n\times sm}$, $X_m= Z_mZ_m^T$.
The matrix $Y_m$ may be numerically rank deficient, and this can be exploited to further decrease
the rank of $Z_m$. We write the eigendecomposition of $Y_m$, $Y_m=W\Sigma W^T$ (with eigenvalues ordered
non-increasingly) and discard only the eigenvalues below
a certain tolerance, that is $\Sigma={\rm diag}(\Sigma_1, \Sigma_2)$, $W=[W_1, W_2]$ with $\|\Sigma_2\|_F \le \epsilon$
(in all our experiments we used $\epsilon=10^{-12}$).
Therefore, we define again
$Y_m\approx \widehat Y \widehat Y^T$, with $\widehat Y = W_1 \Sigma_1^{1/2} \in\mathbb{R}^{sm\times t}$, $t\leq sm$;
in this way,
$\|Y_m-\widehat Y \widehat Y^T\|_F\leq \epsilon$. Hence, we set
$Z_m=V_m\widehat Y\in\mathbb{R}^{sm\times t}$.
We notice that a significant rank reduction in $Y_m$ is an indication 
that all relevant information for generating $X_m$ is actually contained
in a subspace that is much smaller than $\mathbf{K}^\square_m(A,C)$. In other words,
the generated Krylov subspace is not efficient in capturing the solution information
and a much smaller space could have been generated to obtain an approximate solution of comparable accuracy.

\setlength\arrayrulewidth{1pt}
 \begin{table}[htb] \label{proj}
\begin{tabular}{l}
 % \vspace{0.2cm} 
  \textbf{Algorithm 1:} Galerkin projection method for the Lyapunov matrix equation  \\
 %   \hline
 %   \\
    \textbf{Input}: $A\in\mathbb{R}^{n\times n},$ $A$ symmetric and negative definite, $C\in\mathbb{R}^{n\times s}$\\
    \textbf{Output}: $Z_m\in\mathbb{R}^{n\times t}$, $t\leq sm$\\
    \\
  %  \hline
  %  \\
  \textbf{1.} Set $\beta=\|C\|_F$\\
  \textbf{2.} Perform economy-size QR of $C$, $C=V_1 {\pmb \gamma}$. Set $\mathcal{V}_1\equiv V_1$ \\ %, $V_1^TB=\pmb{\gamma}$  \\
  \textbf{3. For} $m=2, 3,\dots,$ till convergence, \textbf{Do}\\
  \textbf{4.} $\qquad$ Compute next basis block $\mathcal{V}_{m}$ and set $V_m=[V_{m-1},\mathcal{V}_m]$ \\
  \textbf{5.} $\qquad$ Update $T_m=V_m^TAV_m$ \\
  \textbf{6.} $\qquad$ \textbf{Convergence check:}\\
  \textbf{6.1} $\qquad$ $\qquad$ Solve $T_mY_m+Y_mT_m+E_1\pmb{\gamma\gamma}^TE_1^T=0$, $\quad E_1\in\mathbb{R}^{ms\times s}$\\
  \textbf{6.2} $\qquad$ $\qquad$ Compute $\|R_m\|_F=\sqrt{2}\|Y_mE_m\pmb{\tau}^T_{m+1,m}\|_F$\\
  \textbf{6.3} $\qquad$ $\qquad$ If $\|R_m\|_F/\beta^2$ is small enough \textbf{Stop} \\ %, otherwise \textbf{Continue}\\
  \textbf{7. EndDo} \\
  \textbf{8.} Compute the eigendecomposition of $Y_m$ and retain $\widehat Y\in\mathbb{R}^{sm\times t}$, $t\leq sm$\\
  \textbf{9.} Set $Z_m=V_m\widehat Y$\\
 % \\
 % \hline
 \end{tabular}  
 \end{table}

Algorithm 1 describes the generic Galerkin
procedure to determine $V_m,Y_m$ and $Z_m$ as $m$ grows, see, e.g., \cite{Simoncini2014}.
Methods thus differ for the choice of the approximation space.
If the block Krylov space $\mathbf{K}^\square_m(A,C)$ is chosen, the block Lanczos
method can be employed in line $4$ of Algorithm 1. 
In exact arithmetic,
\begin{equation}\label{Lanczos}
 \mathcal{V}_{m}\pmb{\tau}_{m+1,m}=A\mathcal{V}_{m-1}-\mathcal{V}_{m-1}\pmb{\tau}_{m,m}-\mathcal{V}_{m-2}\pmb{\tau}_{m-1,m}.
\end{equation}
Algorithm 2 describes this process at iteration $m$, with
$W = A\mathcal{V}_{m-1}$, where the orthogonalization
coefficients $\pmb{\tau}$'s are computed by the modified block Gram-Schmidt procedure (MGS), see, e.g., \cite{Saad2003};
to ensure local orthogonality in finite precision arithmetic, MGS is performed twice (beside each command is the
leading computational cost of the operation).
To simplify the presentation, we assume throughout that the generated basis is full rank. Deflation
could be implemented as it is customary in block methods whenever rank deficiency is detected.

%Since the new basis vector $\mathcal{V}_{m}$ is computed by orthonormalizing $A\mathcal{V}_{m-1}$ with respect to
%$\mathcal{V}_{m-1}$ and $\mathcal{V}_{m-2}$ only, a loss of orthogonality is observed as $m$ grows.
%{\color{blue} However, it is proved in  [cit from Kressner] that this issue just slightly affects bounds on $\|X-X_m\|_2$.
%}
\setlength\arrayrulewidth{1pt}
 \begin{table}[!ht] 
%{\color{blue}
\begin{tabular}{l}
 % \vspace{0.2cm} 
  \textbf{Algorithm 2:} One step of block Lanczos with block MGS  \\
 %   \hline
 %   \\
    \textbf{Input}: $m$, $W$, $\mathcal{V}_{m-2},\mathcal{V}_{m-1}\in\mathbb{R}^{n\times s}$\\
    \textbf{Output}: $\mathcal{V}_m\in\mathbb{R}^{n\times s}$, $\pmb{\tau}_{m-1,m},\pmb{\tau}_{m,m},
    \pmb{\tau}_{m+1,m}\in\mathbb{R}^{s\times s}$\\
    \\
  %  \hline
  %  \\
  %\textbf{1.} Compute $W=A\mathcal{V}_{m-1}$ \\%\hspace{0.5cm} $\leftarrow$ cost depends on $\mathtt{spy}(A)$ \\
  \textbf{1.} Set $\pmb{\tau}_{m-1,m}=\pmb{\tau}_{m,m}=\mathbf{0}$\\
  \textbf{2. For} $l=1,2,$ \textbf{Do}\\
  \textbf{3.} $\qquad$\textbf{For} $i=m-1,m,$ \textbf{Do}\\
  \textbf{3.} $\qquad$ $\qquad$ Compute $\pmb{\alpha}=\mathcal{V}_{i-1}^TW$ \hspace{0.5cm} $\leftarrow$ $(2n-1)s^2$ flops\\
  \textbf{5.} $\qquad$ $\qquad$ Set $\pmb{\tau}_{i,m}=\pmb{\tau}_{i,m}+\pmb{\alpha}$ \hspace{0.5cm} $\leftarrow$ $s^2$ flops\\
  \textbf{6.} $\qquad$ $\qquad$ Compute $W=W-\mathcal{V}_{i-1}\pmb{\alpha}$ \hspace{0.5cm} $\leftarrow$ $2s^2n$ flops\\
  \textbf{7.} $\qquad$ \textbf{EndDo}\\
  \textbf{8.   EndDo} \\
  \textbf{9.} Perform economy-size QR of $W$, $W=\mathcal{V}_m\pmb{\tau}_{m+1,m}$ \hspace{0.1cm} $\leftarrow$ 
    $3ns^2$ flops\\
    % \\
 % \hline
 \end{tabular}  
%}
 \end{table}

We emphasize that only the last $3s$ terms of the basis must be stored,
and the computational cost of Algorithm 2 is fixed with respect to $m$.
 In particular, at each iteration $m$,
 Algorithm 2 costs $\mathcal{O}\left((19n+s)s^2\right)$ flops.
 %plus the multiplication in line 1, whose cost 
 % depends on the sparsity pattern of $A$.

As the approximation space expands, the principal costs of Algorithm 1
are steps 4 and 6.1.  In particular, the computation of the whole matrix
$Y_m$ requires full matrix-matrix operations and a Schur decomposition of the 
coefficient matrix $T_m$, whose costs are $\mathcal{O}\left((sm)^3\right)$ flops.
Clearly, step 6.1 becomes comparable with step 4 in cost for $sm \gg 1$,
for instance if convergence is slow, so that $m\gg 1$.

Step 9 of Algorithm 1 shows that at convergence, the whole basis must be saved to return the factor $Z_m$.
This represents a major shortcoming when convergence is slow, since $V_m$ may require 
large memory allocations.

%%%%%%%%%%%%%%%%%%%%%%%%%%%%%%%%%%%%%%%%%%%%%%%%%%%%%%%%%%%%%%%%%%%%%%%%%%%%%%%%%%%%%%%%%%%%%%%%%%%%%%%%%%%%%%%%%%%%%%

\section{Standard Krylov subspace}\label{Standard Krylov subspace}
For the block space $\mathbf{K}_m^\square(A,C)$,
we devise a new residual norm expression and discuss the two-pass strategy.

%%%%%%%%%%%%%%%%%%%%%%%%%%%%%%%%%%%%%%%%%%%%%%%%%%%%%%%%%%%%%%%%%%%%%%%%%%%%%%%%%%%%%%%%%%%%%%%%%%%%%%%%%%%%%%

\subsection{Computing the residual norm without the whole solution $\mathbf{Y}_m$}\label{SK_residual}
The solution of the projected problem (\ref{reducesLyap}) requires the Schur decomposition of $T_m$.
For real symmetric matrices, the Schur decomposition amounts to the eigendecomposition
$T_m=Q_m\Lambda_mQ_m^T$, $\Lambda_m=\mbox{diag}(\lambda_1,\dots,\lambda_{sm})$,
 and the symmetric block tridiagonal structure of $T_m$ can be exploited so as to use only
${\cal O}((sm)^2)$ flops; see section~\ref{sec:Algorithm} for further details.
Equation (\ref{reducesLyap}) can thus be written as
\begin{equation}\label{diagonalLyap}
 \Lambda_m\widetilde Y+\widetilde Y\Lambda_m+Q_m^{T}E_1\pmb{\gamma\gamma}^TE_1^TQ_m=0, \quad
{\rm where}\quad \widetilde Y:=Q_m^{T}Y_mQ_m. 
\end{equation}
Since $\Lambda_m$ is diagonal, the entries of $\widetilde Y$ can be computed
by substitution \cite[Section~4]{Simoncini2014}, so that
\begin{equation}\label{lyap2sol.}
Y_m = Q_m\widetilde YQ_m^T =
 -\,Q_m \left( \frac{ e_i^T Q_m^{T} E_1\pmb{\gamma\gamma}^TE_1^T Q_m e_j}{\lambda_i+\lambda_j} \right)_{ij} Q_m^{T},
\end{equation}
where $e_k$ denotes the $k$th vector of the canonical basis of $\mathbb{R}^{sm}$.
It turns out that only the quantities within parentheses in (\ref{lyap2sol.}) are needed for the residual norm computation, 
thus avoiding the $\mathcal{O}\left((sm)^3\right)$ cost of recovering $Y_m$.
%The computation of $Y_m$ by (\ref{lyap2sol.}) needs the matrix-matrix product by the full matrix $Q_m\in\mathbb{R}^{sm\times sm}$
%which requires $\mathcal{O}\left((sm)^3\right)\,flops$.

%As already mentioned, the residual norm can be computed as
%%
%\begin{equation}\label{reseqSK.}
% \|R_m\|_F=\|Y_mE_m\pmb{\tau}^T_{m+1,m}\|_F,
%\end{equation}
%%
%where $\pmb{\tau}_{m+1,m}=E_{m+1}^T\underline{T}_mE_{m}$. 
%Recalling the expression in (\ref{reseq1.}) for the residual norm,
%in the next proposition we show how to check convergence without explicitly computing $Y_m$ by (\ref{lyap2sol.}).

\begin{proposition}\label{propSK}
 Let $T_m=Q_m\Lambda_mQ_m^T$ denote the eigendecomposition of $T_m$. Then
 \begin{eqnarray}\label{formula}
\|R_m\|_F^2 &=&2\left(
\|e_1^TS_mD_1^{-1}W_m\|_2^2+\ldots+
\|e_{sm}^TS_mD_{sm}^{-1}W_m\|_2^2\right),
%|e_1^T q_1|^2 (q_1^T D_1^{-1} q_m)^2 +
%|e_2^T q_1|^2 (q_1^T D_2^{-1} q_m)^2 + 
%\ldots +
%|e_m^T q_1|^2 (q_1^T D_m^{-1} q_m)^2 \right ),
\end{eqnarray}
where $S_m=Q_m^TE_1\pmb{\gamma\gamma}^TE_1^TQ_m\in\mathbb{R}^{sm\times sm}$, 
$W_m=Q_m^TE_m\pmb{\tau}^T_{m+1,m}\in\mathbb{R}^{sm\times s}$ and $D_j=\lambda_jI_{sm}+\Lambda_m$ for all 
$j=1,\dots,sm$.
\end{proposition}

\begin{proof}
 Exploiting (\ref{reseq1.}) and the representation formula (\ref{lyap2sol.}) we have
\begin{equation}
\begin{array}{rll}\label{eq1prop} 
\|R_m\|_F^2 &=& 2\displaystyle{\|Y_mE_m\pmb{\tau}^T_{m+1,m}\|_F^2=2
\left\| \left( \frac{ e_i^T Q_m^{T} E_1\pmb{\gamma\gamma}^TE_1^T Q_m e_j}{\lambda_i+\lambda_j} \right)_{ij} Q_m^{T}E_m\pmb{\tau}^T_{m+1,m}\right\|_F^2 
}\\
%&&\\
&=& 2\displaystyle{\sum_{k=1}^s\left\| \left( \frac{ e_i^T S_m e_j}{\lambda_i+\lambda_j} \right)_{ij} W_me_k\right\|_2^2 
.}\\
\end{array}
\end{equation}
For all $k=1,\dots,s$, we can write
{\small
 \begin{equation}
\begin{array}{rll} 
\displaystyle{\left\| \left( \frac{ e_i^T S_m e_j}{\lambda_i+\lambda_j} \right)_{ij} W_me_k\right\|_2^2 
}&=&
\displaystyle{\left(\sum_{j=1}^{sm} \frac{ e_1^T S_m e_j}{\lambda_1+\lambda_j} e_j^T W_me_k\right)^2 +
 \ldots +
\left(\sum_{j=1}^{sm} \frac{ e_{sm}^T S_m e_j}{\lambda_{sm}+\lambda_j} e_j^T W_m e_k\right)^2 }\\ 
&&\\
&=&\displaystyle{ \left(e_1^TS_mD_1^{-1}W_me_k\right)^2+\ldots+\left(e_{sm}^TS_mD_{sm}^{-1}W_me_k\right)^2.}\\
\end{array}
\label{eq2prop}
 \end{equation}}
Plugging (\ref{eq2prop}) into (\ref{eq1prop}) we have
$$\begin{array}{rll}
\|R_m\|_F^2&=&\displaystyle{2\sum_{k=1}^s\sum_{i=1}^{sm}\left(e_i^TS_mD_i^{-1}W_me_k\right)^2=
2\sum_{i=1}^{sm}\sum_{k=1}^s\left(e_i^TS_mD_i^{-1}W_me_k\right)^2}\\
&&\\
&=&\displaystyle{2\sum_{i=1}^{sm}\left\|e_i^TS_mD_i^{-1}W_m\right\|_2^2.}\\   
  \end{array}
$$
 \end{proof}

%Proposition \ref{propSK} shows that only the eigenvalues and the first and last $s$ entries of the eigenvectors of $T_m$ are 
%needed to compute the residual norm. This fact might be exploited by the eigensolver avoiding the computation of the whole
%$Q_m$.

\setlength\arrayrulewidth{1pt}
 \begin{table}[!ht] \label{algSK}
\begin{tabular}{l}
 % \vspace{0.2cm} 

  \textbf{Algorithm 3}: {\tt cTri}  \\
    %\hline
    %\\
    \textbf{Input}: $T_m\in\mathbb{R}^{\ell m\times \ell m},\,\pmb{\gamma},\,\pmb{\tau}_{m+1,m}\in\mathbb{R}^{\ell\times \ell}$ 
($\ell$ is the block size)\\
    \textbf{Output}: $res\;(=\|R\|_F)$\\
    \\
  %  \hline
   % \\
  \textbf{1.} Tridiagonalize $P_m^TT_mP_m=F_m$\\
   \textbf{2.} Compute $F_m=G_m\Lambda_mG_m^T$\\
   \textbf{3.} Compute $E_1^TQ_m=\left(E_1^TP_m\right)G_m$, $E_m^TQ_m=\left(E_m^TP_m\right)G_m$ \\
   \textbf{4.} Compute $S_m=\left(Q_m^TE_1\pmb{\gamma}\right)\left(\pmb{\gamma}^TE_1^TQ_m\right)$ \hspace{0.5cm}
  $\leftarrow$ $(2\ell-1)\ell^2m+(2\ell-1)\ell^2m^2\,flops$ \\
 \textbf{5.} Compute $W_m=\left(Q_m^TE_m\right)\pmb{\tau}^T_{m+1,m}$\hspace{0.5cm} $\leftarrow$ $(2\ell-1)\ell^2m\,flops$\\
 \textbf{6.} Set $res=0$\\
  \textbf{7. For} $i=1,\dots,\ell m$, \textbf{Do}\\
  %\\
  %5. $\qquad$ $\displaystyle{r_j=\sum_{i=1}^m\left(e_i^TS_mS_m^T D_1^{-1} q^{(j)}_m\right)^2}$\\
  \textbf{8.} $\qquad$ Set $D_i=\lambda_iI_{\ell m}+\Lambda_m$\\
  %\\
  \textbf{9.} $\qquad \;res=res+\left\|\left(e_i^TS_m\right)D_i^{-1}W_m\right\|_2^2$\hspace{0.5cm}
  $\leftarrow$ $2\ell^2m+\ell m+\ell \,flops$\\
  \textbf{10. EndDo} \\
  \textbf{11.} Set $res=\sqrt{2}\sqrt{res}$\\
    \\
 % \hline
 \end{tabular}
 \end{table}

%%%%%%%%%%%%%%%%%%%%%%%%%%%%%%%%%
\subsection{The algorithm for the residual norm computation}\label{sec:Algorithm}
Algorithm~3 summarizes the procedure that takes advantage of Proposition~\ref{propSK}. 
%The algorithm shows that 
 Computing the residual norm by (\ref{eq1prop}) has a leading cost of
 $4s^3m^2$ flops for standard Krylov (with $\ell=s$). This should be
compared with the original procedure in
 steps $6.1$ and $6.2$ of Algorithm~1, whose cost is $\mathcal{O}\left(s^3m^3\right)$ flops, with a large constant.
 Proposition \ref{propSK} also shows that only the first and last $\ell$ components of the eigenvectors of $T_m$ are necessary
 in the residual norm evaluation and
  the computation of the complete eigendecomposition $T_m=Q_m\Lambda_mQ_m^T$ may be avoided. To this end, the matrix $T_m$ can be
 tridiagonalized, $P_m^TT_m  P_m=F_m$, explicitly computing only the first and 
last $\ell$ rows of the transformation matrix $P_m$,
 namely $E_1^TP_m$ and $E_m^TP_m$. The eigendecomposition $F_m=G_m\Lambda_mG_m^T$ is computed exploiting the tridiagonal
 structure of $F_m$. The matrices $E_1^TQ_m$ and $E_m^TQ_m$ needed in (\ref{formula}) are then computed as
 $E_1^TQ_m=\left(E_1^TP_m\right)G_m$, $E_m^TQ_m=\left(E_m^TP_m\right)G_m$. 
%can be replaced by Algorithm 3 with a remarkable savings in
%the overall computational cost.
%We also remark that the eigendecomposition of $T_m$
%can exploit the block tridiagonal structure of $T_m$ and use only $\mathcal{O}\left((sm)^2\right)$ flops for determining
%all eigenvalues and eigenvectors.

Once the stopping criterion in step 6.3 of Algorithm 1 is satisfied, the factor $Z_m$ can be finally computed.
Once again, this can be performed without explicitly computing $Y_m$, which requires
the expensive computation $Y_m=Q_m \widetilde Y Q_m^T$. Indeed, the truncation strategy discussed around (\ref{eqn:Y})
can be applied to $\widetilde Y$ by computing the matrix 
$\widecheck Y \in\mathbb{R}^{sm\times t}$, $t\leq sm$ so that $\widetilde Y\approx \widecheck Y \widecheck Y^T$.
This factorization further reduces the overall computational cost,
since only $(2ms-1)tms$ flops are required to compute $Q_m\widecheck Y$, with no loss of information
at the prescribed accuracy. The solution factor $Z_m$ is then computed as $Z_m=V_m\left(Q_m\widecheck Y\right)$.

%%%%%%%%%%%%%%%%%%%%%%%%%%%%%%%%%%%%%%%%%%%%%%%%%%%%%%%%%%%%%%%%%%%%%%%%%%%%%%%%%%%%%%%%%%%%%%%%%%%%%%%%%%%%%%%%%%%%%
%\subsection{Implementation details}\label{Implementation details}
To make fair comparisons with state-of-the-art algorithms that employ LAPACK and SLICOT subroutines
(see Section \ref{Numerical Experiments} for more details), we used a C-compiled mex-code {\tt cTri} to implement Algorithm~3,
making use of LAPACK and BLAS subroutines.
In particular, the eigendecomposition $T_m=Q_m\Lambda_mQ_m^T$ is performed as follows.
The block tridiagonal matrix $T_m$ is tridiagonalized, $P_m^TT_mP_m=F_m$, by the LAPACK subroutine {\tt dsbtrd} that exploits
its banded structure. The transformation matrix $P_m$ is represented as a product of elementary reflectors and only its
first and last $\ell$ rows, $E_1^TP_m$, $E_m^TP_m$, are actually computed.
The LAPACK subroutine {\tt dstevr} is employed to compute the eigendecomposition of the tridiagonal matrix $F_m$.
This routine applies Dhillon's MRRR method \cite{Dhillon1997a} whose main advantage is 
the computation of numerically orthogonal eigenvectors without an explicit orthogonalization procedure.
This feature limits to $\mathcal{O}((\ell m)^2)$ flops the computation of $F_m=G_m\Lambda_m G_m^T \in\RR^{\ell m\times\ell m}$;
 see \cite{Dhillon1997a} for more details. 
 Since the residual norm computation (\ref{formula}) requires the first and last $\ell$ rows
of the eigenvectors matrix $Q_m$, we compute only those components, that is
$E_1^TQ_m=\left(E_1^TP_m\right)G_m$ and $E_m^TQ_m=\left(E_m^TP_m\right)G_m$, 
avoiding the expensive matrix-matrix product $Q_m=P_mG_m$.

%%%%%%%%%%%%%%%%%%%%%%%%%%%%%%%%%%%%%%%%%%%%%%%%%%%%%%%%%%%%%%%%%%%%%%%%%%%%%%%%%%%%%%%%%%%%%%%%%%%%%%%%%%%%%%%%%%%%%%%%%

\subsection{A ``two-pass'' strategy}\label{SK_twopass} 
While the block Lanczos method requires the storage of only $3s$ basis vectors, 
the whole $V_m=[\mathcal{V}_1,\dots,\mathcal{V}_m]\in\mathbb{R}^{n\times sm}$ is needed to compute
the low-rank factor $Z_m$ at convergence (step 9 of Algorithm 1). 
%For large $s$, this becomes very memory demanding after few iterations.
Since
\begin{equation}\label{twopasseq}
Z_m=V_m(Q_m \widecheck Y)=\sum_{i=1}^m\mathcal{V}_iE_i^T (Q_m \widecheck Y),
\end{equation}
we suggest not to store $V_m$ during the iterative process but to perform, at convergence, a second Lanczos
pass computing and adding the rank-$s$ term  in (\ref{twopasseq}) at the $i$th step,
in an incremental fashion.
We point out that the orthonormalization coefficients %required to compute the basis vectors  
are already available in the matrix $T_m$, therefore $\mathcal{V}_i$ is simply computed
by repeating the three-term recurrence (\ref{Lanczos}),
%, that is
%%
%$$ \mathcal{V}_{i}=\left(A\mathcal{V}_{i-1}-\mathcal{V}_{i-2}\pmb{\tau}_{i-1,i}-
%\mathcal{V}_{i-1}\pmb{\tau}_{i,i}\right)\pmb{\tau}_{i+1,i}^{-1},
%$$
which costs $\mathcal{O}\left((4n+1)s^2\right)$ flops plus the multiplication by $A$, making the second Lanczos pass cheaper than
the first one.

%%%%%%%%%%%%%%%%%%%%%%%%%%%%%%%%%%%%%%%%%%%%%%%%%%%%%%%%%%%%%%%%%%%%%%%%%%%%%%%%%%%%%%%%%%%%%%%%%%%%%%%%%%%%%%%%%%%

\section{Extended Krylov subspace}\label{Extended}
%In 2007, the introduction of a solution process based on the Extended Krylov subspace $\mathbf{EK}_m(A,B)$ gave a new life to
%projection methods in the large-scale matrix equations setting.
Rational Krylov subspaces have shown to provide dramatic performance improvements over classical
polynomial Krylov subspaces, because they build spectral information earlier, thus generating a much smaller
space dimension to reach the desired accuracy. The price to pay is that each iteration is more computationally
involved, as it requires solves with the coefficient matrices. The overall CPU time performance thus depends
on the data sparsity of the given problem; we refer the reader to \cite{Simoncini2014} for a thorough discussion.

In this section we show that the enhanced procedure for the residual norm computation can be applied
to a particular rational Krylov based strategy,
the {\it Extended} Krylov subspace method, since also this algorithm relies on a block tridiagonal reduced
matrix when data is symmetric.
Different strategies for building the basis $V_m=[\mathcal{V}_1,\ldots,\mathcal{V}_m]\in\mathbb{R}^{n\times2sm}$ of the 
extended Krylov subspace $\mathbf{EK}_m^\square(A,C)$ can be found in the literature, see, e.g.,
\cite{Jagels2009},\cite{Mertens2015},\cite{Simoncini2007}. An intuitive key fact is that the subspace
expands in the directions of $A$ and $A^{-1}$ at the same time.
In the block case, a natural implementation thus generates two new blocks of vectors at the time, one in
each of the two directions. Starting with $[V_1, A^{-1}V_1]$, the next iterations generate 
the blocks $\mathcal{V}_m^{(1)}, \mathcal{V}_m^{(2)}\in\RR^{n\times s}$ by multiplication by $A$ and solve with $A$, respectively,
and then setting $\mathcal{V}_m = [\mathcal{V}_m^{(1)}, \mathcal{V}_m^{(2)}]\in\mathbb{R}^{n\times 2s}$.
As a consequence, the block Lanczos procedure
described in Algorithm 2 can be employed with $W = [ A\mathcal{V}^{(1)}_{m-1}, A^{-1}\mathcal{V}^{(2)}_{m-1}]$ (with $2s$ columns).
%{\color{blue} Although the LU factors of $A$ can be computed once for all at the beginning of the iterative process as suggested
%in \cite{Simoncini2007}, the linear solves dramatically increase the computational cost of the basis construction if compared
%to the standard case $\mathbf{K}_m^\square(A,C)$ where only matrix-vector products are required.
%See Section \ref{Numerical Experiments}.}
The orthogonalization process determines the coefficients of the symmetric block tridiagonal matrix $H_m$ with blocks of size $2s$,
$$ H_m=\left(\begin{array}{lll}       
        \pmb{\vartheta}_{11} &\quad \pmb{\vartheta}_{12} &   \\
        \pmb{\vartheta}_{21} & \quad \pmb{\vartheta}_{22} & \quad \pmb{\vartheta}_{23}   \\
         &  \ddots &   \ddots \quad \quad  \ddots   \\
        &   \quad \ddots  &\quad \ddots \qquad \pmb{\vartheta}_{m-1,m} \\
    & &  \pmb{\vartheta}_{m,m-1}\quad  \pmb{\vartheta}_{m,m}  \\
       \end{array}\right)\in\mathbb{R}^{2sm\times 2sm},
$$
such that
$\mathcal{V}_{m}\pmb{\vartheta}_{m+1,m}=[A\mathcal{V}_{m-1}^{(1)},A^{-1}\mathcal{V}_{m-1}^{(2)}]-\mathcal{V}_{m-1}\pmb{\vartheta}_{m,m}
-\mathcal{V}_{m-2}\pmb{\vartheta}_{m-1,m}$.
The coefficients $\pmb\vartheta$'s correspond to the $\pmb\tau$'s in Algorithm 2, however as opposed
to the standard Lanczos procedure, $H_m \ne T_m=V_m^T A V_m$. Nonetheless, a recurrence
can be derived to compute the columns of $T_m$ from those of $H_m$ during the
iterations; see \cite[Proposition~3.2]{Simoncini2007}. The computed $T_m$ is block tridiagonal,
with blocks of size $2s$, and
this structure allows us to use the same approach followed for the block standard Krylov method
as relation (\ref{reseq1.}) still holds. Algorithm 3 can thus be adopted to compute the residual norm also 
in the extended Krylov approach with $\ell=2s$.
%The only difference may be in the computation of the matrix-matrix product 
%$W_m=Q_m^TE_m\pmb{{\tau}}_{m+1,m}^T\in\mathbb{R}^{2ms\times 2s}$. Indeed, 
Moreover, it is shown in \cite{Simoncini2007} that
the off-diagonal blocks of $T_m$ %, namely $\pmb{\tau}_{i,j}\in\mathbb{R}^{2s\times 2s}$, $j=1,\ldots,m$, $i=j\pm 1$, 
have a zero lower $s\times 2s$ block, that is
$$
{\pmb\tau}_{i,i-1}=
\begin{bmatrix}
{\overline{\pmb\tau}}_{i,i-1}\\
0
\end{bmatrix},\quad 
\overline{\pmb\tau}_{i,i-1}\in\mathbb{R}^{s\times 2s},\; i=1,\ldots,m .
$$
This observation can be exploited in the computation of the residual norm as
$$ \|R_m\|_F=\sqrt{2}\|Y_mE_m\pmb{\tau}^T_{m+1,m}\|_F=\sqrt{2}\|Y_mE_m\pmb{\overline{\tau}}^T_{m+1,m}\|_F,$$
and $\pmb{\overline{\tau}}_{m+1,m}$ can be passed as an input argument to {\tt cTri} instead of the whole 
$\pmb{\tau}_{m+1,m}$.

The extended Krylov subspace dimension grows faster than the standard one as it is augmented 
by $2s$ vectors per iteration. In general, this does not create severe storage difficulties as the extended Krylov approach 
exhibits faster convergence than standard Krylov in terms of number of iterations.
However, for hard problems the space may still become too large to be stored, especially for large $s$.
In this case, a two-pass-like strategy may be appealing. To avoid the occurrence of $sm$ new system solves
with $A$, however,
it may be wise to still store the second blocks, $\mathcal{V}_i^{(2)}$, $i=1, \ldots, m$, and only save
half memory allocations, those corresponding to the matrices $\mathcal{V}_i^{(1)}$, $i=1, \ldots, m$.

Finally, we remark that if we were to use more general rational Krylov subspaces, which
use rational functions other than $A$ and $A^{-1}$ to generate the space \cite{Simoncini2014}, the
projected matrix $T_m$ would lose the convenient block tridiagonal structure, so that the
new strategy would not be applicable.
%Unfortunately, in the extended procedure, the execution time is highly influenced by the cost of solving with $A$ 
%so that a second basis construction may not pay off.
%A simple remedy might be the following. During the iterative process all the basis blocks $\mathcal{V}_i^{(2)}$,
%$i=1,\ldots,m$ are stored.
%At convergence, the missing basis blocks $\mathcal{V}_i^{(1)}$ are computed in a second Lanczos pass by Algorithm 2 and
%the matrix $Z_m$ is computed on-the-fly.
%This strategy does not increase the number of the costly linear system solves since in the second pass only the 
%multiplication by $A$ is requested.
%Moreover, the memory requirement is halved as only $ms+3s$ basis vectors are stored.
%to store only the $sm$ basis vectors in the $A^{-1}$ direction and other $3s$ vectors
%in the $A$ one during the iterative process.
%In a second Lanczos pass, the missing basis vectors can be computed exploiting the orthogonalization coefficients 
%collected in $T_m$.
%This procedure halves the memory requirement of the whole process with no increment in the number of the costly linear
%system solves.  
%the moderate number of iterations achieved does not necessarily imply that a low
%memory requirement is needed. In fact, 
%This is the reason why the Extended Krylov subspace method is applied to problems 
%with ``very'' low rank right-hand, that is $s\leq 10$.

%%%%%%%%%%%%%%%%%%%%%%%%%%%%%%%%%%%%%%%%%%%%%%%%%%%%%%%%%%%%%%%%%%%%%%%%%%%%%%%%%%%%%%%%%%%%%%%%%%%%%%%%%%%%%%
\section{The case of the Sylvester equation}\label{Sylvester_intro}
The strategy presented for the symmetric Lyapunov equation (\ref{Lyapeq.}) can be extended to
the Sylvester equation
\begin{equation}\label{Sylveq.}
 AX+XB+C_1C_2^T=0,\quad A\in\mathbb{R}^{n_1\times n_1},B\in\mathbb{R}^{n_2\times n_2},C_1\in\mathbb{R}^{n_1\times s},
 C_2\in\mathbb{R}^{n_2\times s}, 
\end{equation}
where the coefficient matrices $A,B$ are both symmetric and negative definite while 
$C_1,C_2$ are tall, that is $s\ll n_1,n_2$. 

%%%%%%%%%%%%%%%%%%%%%%%%%%%%%%%%%%%%%%%%%%%%%%%%%%%%%%%%%%%%%%%%%%%%%%%%%%%%%%%%%%%%%%%%%%%%%%%%%%%%%%%%%%%%%%%%%%%

\subsection{Large $\mathbf{A}$ and large $\mathbf{B}$}\label{The Sylvester case}
We consider the case when both $A$ and $B$ are large and sparse matrices.
If their eigenvalue distributions satisfy certain hypotheses, 
the singular values of the 
nonsymmetric solution $X\in\mathbb{R}^{n_1\times n_2}$ to (\ref{Sylveq.}) exhibit a fast decay, and a low-rank approximation
$\widetilde X=Z_1Z_2^T$ to $X$ can be sought; see, e.g., \cite[Th. 2.1.1]{Sabino2006}, \cite[Section 4.4]{Simoncini2014}.

Projection methods seek an approximate solution $X_m\in\mathbb{R}^{n_1\times n_2}$ to (\ref{Sylveq.})
of the form $X_m=V_mY_mU_m^T$ where the orthonormal columns of $V_m$ and $U_m$ span suitable subspaces $\mathcal{K}_m$
and $\mathcal{C}_m$ respectively\footnote{The space dimensions of ${\cal K}_m$ and ${\cal C}_m$ are not
necessarily equal, we limit our discussion to the same dimension for simplicity
of exposition.}. The construction of two approximation spaces is thus requested and,
for the sake of simplicity, we limit our discussion to the 
standard Krylov method, that is $\mathcal{K}_m=\mathbf{K}^\square_m(A,C_1)$ and $\mathcal{C}_m=\mathbf{K}^\square_m(B,C_2)$,
with obvious generalization to the extended Krylov subspace.
As in the Lyapunov case, $Y_m$ is computed by imposing
a Galerkin condition on the residual matrix $R_m:=AX_m+X_mB+C_1C_2^T$, that is
\begin{equation}\label{resSylv2}
 V_m^TR_mU_m=0.
\end{equation}
We assume $C_1=V_1\pmb{\gamma}_1$, $C_2=U_1\pmb{\gamma}_2$ for some nonsingular 
$\pmb{\gamma}_1,\pmb{\gamma}_2\in\mathbb{R}^{s\times s}$, and a similar discussion 
to the one presented in Section \ref{Projection methods} shows that condition (\ref{resSylv2})
is equivalent to solving the reduced Sylvester problem
\begin{equation}\label{reducedSylv}
 T_mY_m+Y_mJ_m+E_1\pmb{\gamma}_1\pmb{\gamma}_2^TE_1^T=0,
\end{equation}
where $T_m:=V_m^TAV_m$, $J_m:=U_m^TBU_m = ({\pmb\iota}_{ij})$. Similarly to the Lyapunov case, computing the
eigendecompositions $T_m=Q_m\Lambda_mQ_m^T$, $\Lambda_m=\mbox{diag}(\lambda_1,\dots,\lambda_{sm})$, and 
$J_m=P_m\Upsilon_m P_m^T$, $\Upsilon_m=\mbox{diag}(\upsilon_1,\dots,\upsilon_{sm})$, 
the solution $Y_m$ to (\ref{reducedSylv}) can be written as
\begin{equation}\label{Sylvformula}
 Y_m=Q_m\widetilde YP_m^T=-\,Q_m 
\left( \frac{ e_i^T Q_m^{T} E_1\pmb{\gamma}_1\pmb{\gamma}_2^TE_1^T P_m e_j}{\upsilon_i+\lambda_j} \right)_{ij} P_m^{T}.
\end{equation}
%%
%The matrix $Y_m$ is computed at each iteration as 
The last $s$ rows and columns of $Y_m$ are employed in the residual norm calculation.
Indeed, letting $\underline{T}_m=V_{m+1}^TAV_m$ and $\underline{J}_m=U_{m+1}^TBU_m$, it can be shown that
\begin{eqnarray} \label{resSylv}
  \|R_m\|_F^2&=&\|\underline{T}_mY_m\|_F^2+\|Y_m\underline{J}_m^T\|_F^2 =
  \|\pmb{\tau}_{m+1,m} E_m^TY_m\|_F^2+ \|Y_mE_m\pmb{\iota}_{m+1,m}^T\|_F^2 ,
 \end{eqnarray}
where $\pmb{\tau}_{m+1,m}=E_{m+1}^T\underline{T}_mE_m\in\mathbb{R}^{s\times s}$ and
$\pmb{\iota}_{m+1,m}=E_{m+1}^T\underline{J}_mE_m\in\mathbb{R}^{s\times s}$, 
see, e.g., \cite{Simoncini2014},\cite{Breiten2014}.

The same arguments of Section \ref{SK_residual} can be applied to the factors in (\ref{resSylv}) 
leading to Algorithm 4 for the computation of the residual norm without explicitly assembling the matrix $Y_m$.
The eigendecompositions in step 1 are not fully computed. In particular, only the spectrum and the first and last $\ell$ 
components of the eigenvectors of $T_m$ and $J_m$ are explicitly computed following the strategy 
presented in Section \ref{sec:Algorithm}.
\setlength\arrayrulewidth{1pt}
 \begin{table}[!ht] \label{algSKSylv}
\begin{tabular}{l}
  %\vspace{0.2cm} 
    \textbf{Algorithm 4}: Computing the residual norm for $A$ and $B$ large  \\
  %  \hline
  %  \\
    \textbf{Input}: $T_m,\,J_m\in\mathbb{R}^{\ell m\times \ell m},\,\pmb{\gamma}_1,\pmb{\gamma}_2,\,
    \pmb{\tau}_{m+1,m},\,\pmb{\iota}_{m+1,m}\in\mathbb{R}^{\ell \times \ell}$\\
    \textbf{Output}: $res\;(=\|R_m\|_F)$\\
   % \\
    %\hline
    \\
  \textbf{1.} Compute $T_m=Q_m\Lambda_mQ_m^T$ and $J_m=P_m\Upsilon_m P_m^T$\\
  \textbf{2.} Compute $S_m:=\left(Q_m^{T} E_1\pmb{\gamma}_1\right)\left(\pmb{\gamma}_2^TE_1^T P_m\right)$\\
  \textbf{3.} Compute $F_m:=\left(Q_m^{T} E_m\right)\pmb{\tau}_{m+1,m}^T$, 
  $G_m:=\left(P_m^{T} E_m\right)\pmb{\iota}_{m+1,m}^T$\\
  \textbf{4.} Set $res=0$\\
  \textbf{5. For} $i=1,\dots,\ell m$, \textbf{Do}\\
  %\\
  %5. $\qquad$ $\displaystyle{r_j=\sum_{i=1}^m\left(e_i^TS_mS_m^T D_1^{-1} q^{(j)}_m\right)^2}$\\
  %\\
  \textbf{6.} $\qquad$ Set $D_i':=\upsilon_iI_{\ell m}+\Lambda_m$ and  $D_i'':=\lambda_iI_{\ell m}+\Upsilon_m$\\
  \textbf{7.} $\qquad \;res=res+\left\|e_i^TS_mD_i'^{-1}G_m\right\|_2^2+\left\|e_i^TS_m^TD_i''^{-1}F_m\right\|_2^2$\\
  \textbf{8. EndDo} \\
  \textbf{11.} Set $res=\sqrt{res}$\\
 % \\
 % \hline
 \end{tabular}  
 \end{table}

 At convergence, the matrix $Y_m$ can be computed by (\ref{Sylvformula}). 
 Also in the Sylvester problem the matrix $Y_m$ may be numerically singular. 
%, and a low-rank approximation strategy could be performed. 
In this case, factors $\widehat Y_{1},\widehat Y_{2}\in\mathbb{R}^{sm\times t}$, $t\leq sm$, such that 
 $\|\widetilde Y -\widehat Y_{1}\widehat Y_{2}^T\|_F\leq\epsilon$ can be computed via the
 truncated singular value decomposition of the nonsymmetric matrix 
 $\widetilde Y$.
 % truncated singular values decomposition of the nonsymmetric matrix 
 %$\widetilde Y$ can be computed, $\widetilde Y\approx \widehat Y_{L}\widehat Y_{R}^T$, 
 %$\widehat Y_{L},\widehat Y_{R}\in\mathbb{R}^{sm\times t}$, $t\leq sm$, such that 
 %$\|\widetilde Y -\widehat Y_{L}\widehat Y_{R}^T\|_F\leq\epsilon$. 
The low-rank factors $Z_1,Z_2$ of $X_m$, $X_m\approx Z_1Z_2^T$, are then computed as $Z_1=V_m\left(Q_m\widehat Y_1\right)$ and
$Z_2=U_m\left(P_m\widehat Y_2\right)$.

If equation (\ref{Sylveq.}) is solved by the standard Krylov method, the two-pass strategy presented
 in Section \ref{SK_twopass} can be easily adapted to the Sylvester case. Indeed,
 denoting by $V_m=[\mathcal{V}_1,\dots,\mathcal{V}_m]$ and $U_m=[\mathcal{U}_1,\dots,\mathcal{U}_m]$,
 the low-rank factors $Z_1$ and $Z_2$ can be written as 
$$Z_1=V_m\left(Q_m\widehat Y_{1}\right)=\sum_{i=1}^m\mathcal{V}_iE_i^T\left(Q_m\widehat Y_{1}\right),\qquad 
Z_2=U_m\left(P_m\widehat Y_{2}\right)=
\sum_{i=1}^m\mathcal{U}_iE_i^T\left(P_m\widehat Y_{2}\right).$$
As in the Lyapunov case, the factors $Z_1$, $Z_2$ can be computed in a second Lanczos pass since the terms 
$\mathcal{V}_iE_i^T\left(Q_m\widehat Y_{1}\right)$ and $\mathcal{U}_iE_i^T\left(P_m\widehat Y_{2}\right)$ do not require the
whole basis to be available.
Therefore, for the Sylvester problem (\ref{Sylveq.}), the ``two-pass'' strategy allows us to store only $6s$ basis vectors,
$3s$ vectors for each of the two bases.

%%%%%%%%%%%%%%%%%%%%%%%%%%%%%%%%%%%%%%%%%%%%%%%%%%%%%%%%%%%%%%%%%%%%%%%%%%%%%%%%%%%%%%%%%%%%%%%%%%%%%%%%%%%
\subsection{Large $\mathbf{A}$ and small $\mathbf{B}$}\label{LargeAsmallB}
In some applications, such as the solution of eigenvalues problems \cite{Watkins2007} or boundary value problems with 
separable coefficients \cite{Wachspress1963},
the matrices $A$ and $B$ in (\ref{Sylveq.}) could have very different dimensions. In particular, one of them,
for instance, $B$, could be of moderate size, that is $n_2\ll 1000$. 
In this case, the projection method presented in Section \ref{The Sylvester case}
can be simplified. Indeed, a reduction of the matrix $B$ becomes unnecessary, so that
  a numerical solution $X_m$ to (\ref{Sylveq.}) of the form $X_m=V_mY_m$ is sought,
where the columns of $V_m$ span $\mathcal{K}_m=\mathbf{K}^\square_m(A,C_1)$, as before.
%where $\mathcal{K}_m=\mbox{Range}(V_m)$ is used to project $A$.
%For the sake of simplicity, we consider $\mathcal{K}_m=\mathbf{K}_m(A,C_1)$ and $V_m$ having orthonormal columns.
%The matrix $Y_m$ is computed by imposing a Galerkin condition, different from (\ref{resSylv2}),
The Galerkin condition on the residual matrix $R_m:=AX_m+X_mB+C_1C_2^T$ thus becomes
\begin{equation}\label{orthLargeAsmallB}
V_m^TR_m=0, 
\end{equation}
see \cite[Section 4.3]{Simoncini2014} for more details. The procedure continues as in the previous
cases, taking into account that the original problem is only reduced ``from the left''.
%
%Substituting $R_m$ in (\ref{orthLargeAsmallB}) and 
Assuming $C_1=V_1\pmb{\gamma}_1$, we obtain
% for some nonsingular $\pmb{\gamma}_1\in\mathbb{R}^{s\times s}$, so that 
%$V_m^TC_1=E_1\pmb{\gamma}_1$, the orthogonality condition (\ref{orthLargeAsmallB}) is equivalent to 
%%
$$ 0=V_m^TAX_m+V_m^TX_mB+V_m^TC_1C_2^T=\left(V_m^TAV_m\right)Y_m+\left(V_m^TV_m\right)Y_mB+E_1\pmb{\gamma}_1C_2^T,$$
that is
\begin{equation}\label{reducedLargeASmallB}
 T_mY_m+Y_mB+E_1\pmb{\gamma}_1C_2^T=0.
\end{equation}
Computing the eigendecompositions $T_m=Q_m\Lambda_mQ_m^T$, $\Lambda_m=\mbox{diag}(\lambda_1,\ldots,\lambda_{sm})$
and $B=P\Upsilon P^T$, $\Upsilon=\mbox{diag}(\upsilon_1,\ldots,\upsilon_{n_2})$, the solution matrix $Y_m$ to (\ref{reducedLargeASmallB})
can be written as
\begin{equation}\label{formulaLargeAsmallB}
Y_m=Q_m\widetilde YP^T=-Q_m\left(\frac{Q_m^TE_1\pmb{\gamma}_1C_2^TP}{\lambda_i+\upsilon_j}\right)_{ij}P^T. 
\end{equation}
%%
%If the block Lanczos procedure is employed in the basis construction, 
As before, the block tridiagonal structure of $T_m$ 
can be exploited in the eigendecomposition computation $T_m=Q_m\Lambda_mQ_m^T$, while
the eigendecomposition $B=P\Upsilon P^T$ is computed
once for all at the beginning of the whole process.
% can be unstructured. Nonetheless, the computation of its eigendecomposition does not represent a shortcoming
%since $B$ is supposed to have small dimensions. Moreover, the eigendecomposition $B=P\Phi P^T$ can be performed once for 
%all at the beginning of the iterative process and the matrix-matrix product $\pmb{\gamma}_1C_2^TP$ as well.

%The matrix $Y_m$ is computed at each iteration since 
The expression of the residual norm simplifies as
% last $s$ rows of $Y_m$ are employed in the residual norm calculation as
%Indeed, it can be shown that
$\|R_m\|_F =\|Y_m^TE_m^T\pmb{\tau}_{m+1,m}^T\|_F$.
To compute this norm  without assembling the whole matrix $Y_m$,
a slight modification of Algorithm 3 can be implemented.
The resulting procedure is summarized in Algorithm~5 where only selected entries of
the eigenvector matrix $Q_m$ in step 1 are computed;
see the corresponding strategy in Section \ref{sec:Algorithm}.

\setlength\arrayrulewidth{1pt}
 \begin{table}[!ht] 
\begin{tabular}{l}
 % \vspace{0.2cm} 

  \textbf{Algorithm 5}: Computing the residual norm for $A$ large and $B$ small  \\
    %\hline
    %\\
    \textbf{Input}: $T_m\in\mathbb{R}^{\ell m\times \ell m},\,\pmb{\tau}_{m+1,m}\in\mathbb{R}^{\ell\times \ell}$,
    $P^TC_2\pmb{\gamma}_1^T\in\mathbb{R}^{n_2\times \ell}$, $\{\upsilon_i\}_{i=1,\ldots,n_2}$ \\
    \textbf{Output}: $res\;(=\|R\|_F)$\\
    \\
  %  \hline
   % \\
  \textbf{1.} Compute $T_m=Q_m\Lambda_mQ_m^T$\\
  \textbf{2.} Compute $S_m=\left(P^TC_2\pmb{\gamma}_1^T\right)\left(E_1^TQ_m\right)$\\
 \textbf{3.} Compute $W_m=\left(Q_m^TE_m\right)\pmb{\tau}^T_{m+1,m}$ \\
 \textbf{4.} Set $res=0$\\
  \textbf{5. For} $i=1,\dots,n_2$, \textbf{Do}\\
  %\\
  %5. $\qquad$ $\displaystyle{r_j=\sum_{i=1}^m\left(e_i^TS_mS_m^T D_1^{-1} q^{(j)}_m\right)^2}$\\
  \textbf{6.} $\qquad$ Set $D_i=\upsilon_iI_{\ell m}+\Lambda_m$\\
  %\\
  \textbf{7.} $\qquad \;res=res+\left\|\left(e_i^TS_m\right)D_i^{-1}W_m\right\|_2^2$\\
  \textbf{8. EndDo} \\
  \textbf{9. Set $res=\sqrt{res}$} \\
  \\
 % \hline
 \end{tabular}
 \end{table}

 %At convergence, the matrix $Y_m$ can be computed by (\ref{formulaLargeAsmallB}) and the same truncation strategy
 %discussed in the previous Section can be employed to reduce the rank of the final numerical solution. In particular,
A reduced rank approximation to the solution $Y_m$ obtained by (\ref{formulaLargeAsmallB}) is given as
 $\widetilde Y\approx \widehat Y_{1}\widehat Y_{2}^T$, so that the low rank factors $Z_1$, $Z_2$ are computed as
 $Z_1=V_m\left(Q_m\widehat Y_1\right)$ and $Z_2=P\widehat Y_2$. A two-pass strategy can again be employed to avoid
storing the whole matrix $V_m$.
%Moreover, if the Standard Krylov method is adopted, the storing of the 
% whole basis $V_m$ can be omitted and the factor $Z_1$ can be computed in a second Lanczos pass.
 
%%%%%%%%%%%%%%%%%%%%%%%%%%%%%%%%%%%%%%%%%%%%%%%%%%%%%%%%%%%%%%%%%%%%%%%%%%%%%%%%%%%%%%%%%%%%%%%%%%%%%%%%%%%%%%%%

\section{Numerical experiments}\label{Numerical Experiments}
In this section some numerical examples illustrating the enhanced  algorithm are reported.
All results were obtained with Matlab R2015a on a Dell machine with two 2GHz processors and 128 GB of RAM.

The standard implementation of projection methods (Algorithm 1) and the proposed enhancement, 
where lines $6.1$ and $6.2$ of Algorithm 1 are replaced by Algorithm 3, are compared.
For the standard implementation, different decomposition based solvers for 
line $6.1$ in Algorithm 1 are considered: The Bartels-Stewart
algorithm (function {\tt lyap}), one of its variants ({\tt lyap2})\footnote{The function 
{\tt lyap2} was slightly modified to exploit the orthogonality of the eigenvectors matrix.},
 and the Hammarling method ({\tt lyapchol}).
All these algorithms make use of SLICOT or LAPACK subroutines. 

Examples with a sample of small values of the rank $s$
of $C_1C_2^T$  are reported.
 %(as it is customary in the solution of large-scale Lyapunov and 
%Sylvester equations with low-rank $C_1C_2^T$ by projection and ADI type methods.}
%, otherwise the underlying approximation space would grow too quickly. Indeed, at each iteration, $s$ ($2s$)
%new basis vectors are added to the current space $\mathbf{K}^\square_m$ ($\mathbf{EK}^\square_m$). }
In all our experiments the convergence tolerance on the relative residual norm is $\mathtt{tol}=10^{-6}$.

%In all our experiments the stopping criterion is based on the relative residual norm. Indeed, the backward error analysis 
%suggested in \cite{Simoncini2007} can not be performed since it requires the norm of the solution $Y_m$ to (\ref{reducesLyap}).

\vskip 0.1in
\begin{example} \label{blockSK}
{\rm
In the first example, the block standard Krylov approach is tested for solving the
Lyapunov equation $AX+XA + CC^T=0$. We consider %the Lyapunov equation (\ref{Lyapeq.})
$A\in\mathbb{R}^{n\times n}$, $n=21904$ stemming from the discretization by centered finite differences
of the differential operator
$$ {\cal L}(u)=(e^{-xy}u_x)_x+(e^{xy}u_y)_y,$$
on the unit square with zero Dirichlet boundary conditions, 
while $C=\mathtt{rand}\,(n,s)$, $s=1,\,4,\,8$, that is the entries of $C$ are
random numbers uniformly distributed in the interval $(0,1)$. $C$ is then normalized, $C=C/\|C\|_F$.
%We point out that, in this example, $s=4$ and therefore, $4$ basis vectors are added to $\mathbf{K}_m(A,C)$ at each iteration.
Table \ref{ex1.1} (left) reports the CPU time (in seconds) needed for evaluating
the residual norm (time res) and for completing the whole procedure (time tot).
Convergence is checked at each iteration. 
For instance, for $s=1$, using {\tt lyapchol} as inner solver the solution process takes $38.51$ secs, $36.51$ of which are used
for solving the inner problem of step 6.1. If we instead use {\tt cTri}, the factors of $X_m$ are
determined in $7.25$ seconds, only $4.42$ of which
are devoted to evaluating the residual norm. Therefore, $87.9$\% of the  residual computation CPU time is saved,
leading to a $81.2$\% saving for the whole procedure. An explored device to mitigate the residual norm computational cost
is to check the residual only periodically.  In the right-hand side of
Table~\ref{ex1.1} we report the results in case the residual norm is computed every 10
iterations.

 Table~\ref{tab1.1} shows that the two-pass strategy of Section~\ref{SK_twopass}
drastically reduces the memory requirements of the solution process, as already observed in \cite{Kressner2008},
at a negligible percentage of the total execution time.
% since only $3s$ basis vectors need to be stored.
%The second Lanczos pass to build the final solution factor requires a negligible percentage of the total execution time 
%(see Table 5.2) and our memory-conserving strategy thus really pays off.
\setlength\arrayrulewidth{0.4pt}

\begin{table}[htb]
{\scriptsize
 %\captionsetup{justification=centering,font=scriptsize ,labelfont=bf,textfont=it}
\begin{center}
\caption{Example \ref{blockSK}. CPU times and gain percentages. Convergence is checked every $d$ iterations.
Left: $d=1$. Right: $d=10$. \label{ex1.1}}
\begin{tabular}{|r|r c r c| r c r c|}
\hline
 & time res & gain &time tot & gain & time res & gain &time tot & gain\\
 & (secs) &       & (secs)&    & (secs)&  & (secs) & \\
 \hline
 &\multicolumn{4}{|c|}{$s=1$ (444 its)} & \multicolumn{4}{ |c|}{$s=1$} \\
 \hline
 {\tt lyap} & $42.36$ & $89.5$\%&$45.18$&$83.9$\%&
 $4.78$ & $89.7$\%&$7.87$&$52.9$\%\\
%\hline
 {\tt lyapchol} & $36.51$ & $87.9$\% &$38.51$ &$81.2$\% &
 $4.27$ & $88.5$\% &$7.59$ &$51.25$\%  \\
%\hline
 {\tt lyap2}  & $34.27$&$87.1$\%& $37.07$&$80.4$\%&
 $3.85$ & $87.2$\% &$7.14$ &$48.1$\%  \\
% \hline
 {\tt cTri} &{\bf 4.42} &$\curvearrowswne$ &$\mathbf{7.25}$&$\curvearrowswne$&
 $\mathbf{0.49}$ &$\curvearrowswne$ &$\mathbf{3.70}$&$\curvearrowswne$\\
\hline 
&\multicolumn{4}{|c|}{$s=4$ (319 its)}& \multicolumn{4}{| c|}{$s=4$} \\
\hline
 {\tt lyap} & $819.02$ & $96.4$\%&$825.44$&$95.6$\% &
 $88.52$ & $96.6$\%&$95.60$&$91.65$\%  \\
%\hline
 {\tt lyapchol} & $213.87$ & $86.1$\% &$220.51$ &$83.6$\% & 
 $21.38$ & $86.1$\% &$26.83$ &$70.2$\%  \\
%\hline
 {\tt lyap2}  & $212.99$&$86.0$\%& $219.34$&$83.5$\%&  
 $20.28$&$85.3$\%& $27.65$&$71.1$\% \\
% \hline
 {\tt cTri} &{\bf 29.78} &$\curvearrowswne$ &{\bf 36.21}&$\curvearrowswne$& 
 $\mathbf{2.97}$ &$\curvearrowswne$ &$\mathbf{7.98}$&$\curvearrowswne$\\
\hline

&\multicolumn{4}{|c|}{$s=8$ (250 its)}& \multicolumn{4}{| c|}{$s=8$} \\
\hline
 {\tt lyap} &$2823.31$& $97.9$\%&$2836.29$&$97.6$\% &
 $305.11$ & $98.2$\%&$313.49$&$95.8$\%  \\
%\hline
 {\tt lyapchol} & $415.42$& $85.7$\% &$427.21$ &$84.1$\% &$38.94$ & $85.7$\%&$46.96$ &$71.8$\%  \\
%\hline
 {\tt lyap2}  & $424.23$&$86.0$\%& $435.90$&$84.4$\%&$41.39$&$86.5$\%& $49.15$&$73.1$\% \\
% \hline
{\tt cTri} &{\bf 59.25} &$\curvearrowswne$ &{\bf 67.89}&$\curvearrowswne$&$\mathbf{5.56}$&$\curvearrowswne$&
$\mathbf{13.22}$&$\curvearrowswne$\\
\hline 

\end{tabular}
\end{center}
}
\end{table}

\setlength\arrayrulewidth{0.4pt}
\begin{table}[ht!]
\begin{center}
 %\captionsetup{justification=centering,font=scriptsize ,labelfont=bf,textfont=it}
\caption{Example \ref{blockSK}. Memory requirements with and without full storage, and CPU time of the second Lanczos sweep.
 \label{tab1.1}}
\begin{tabular}{|rrr r r r|} \hline
  &  &  & memory   &   reduced  & CPU time  \\
  &  &  &  whole $V_m$  & mem. alloc. &  (secs) \\
 $n$&$s$ &$m$ & $s\cdot m$ &  $3s$& \\
\hline
21904 & 1 & 444  & 444 & 3 & 1.44 \\
 21904 & 4 & 319 & 1276 & 12 & 2.35 \\
 21904& 8& 250&2000& 24& 3.74\\
 \hline
\end{tabular}
\end{center}
\end{table}
}
\end{example}
\vskip 0.1in

\begin{example}\label{RAIL}
{\rm
 The RAIL benchmark problem
 \footnote{\scriptsize{\tt http://www.simulation.uni-freiburg.de/downloads/benchmark/Steel\%20Profiles\%20\%2838881\%29}}
 solves the generalized Lyapunov equation
\begin{equation}\label{exeq.}
 AXE+EXA+CC^T=0,
\end{equation}
where $A,E\in\mathbb{R}^{n\times n}$, $n=79841$, $C\in\mathbb{R}^{n\times s}$, $s=7$. Following the discussion in 
\cite{Simoncini2007}, equation (\ref{exeq.}) can be treated as a standard Lyapunov equation for $E$ symmetric and
positive definite.
This is a recognized hard problem for the standard Krylov subspace, therefore the extended Krylov subspace is applied,
and convergence is checked at each iteration.  Table~\ref{tab2} collects 
the results. 
In spite of the $52$ iterations needed to converge, the space dimension is large, that is
 $\mbox{dim}\left(\mathbf{EK}^\square_m(A,C)\right)=728$ and the memory-saving strategy of Section~\ref{Extended} may be attractive;
it was not used for this specific example, but it can be easily implemented.
 The gain in the evaluation of the residual norm is still remarkable, but
% the expensive Extended Lanczos procedure makes our strategy 
less impressive from the global point of view. 
 Indeed, the basis construction represents the majority of the computational efforts;
 in particular, the linear solves $A^{-1}\mathcal{V}_i^{(2)}$, $i=1,\ldots,52$, required 17.60 seconds.
 
\setlength\arrayrulewidth{0.4pt}
 \begin{table}[!ht]
 \centering
  %\captionsetup{justification=centering,font=scriptsize ,labelfont=bf,textfont=it}
\caption{Example \ref{RAIL}. CPU times and gain percentages. \label{tab2}}
\begin{tabular}{|r| r c r c|}
 \hline
 & time res  & gain &time tot & gain \\
 &  (secs) & gain & (secs)&  \\
 \hline
 {\tt lyap} & $11.25$ &$75.9$\%  &$75.53$ &$7.7$\%  \\
%\hline
 {\tt lyapchol} & $6.05$ & $55.2$\% &$70.76$ &$1.5$\%  \\
%\hline
 {\tt lyap2}  & $6.68$&$59.4$\%& $73.01$&$4.5$\% \\
% \hline
 {\tt cTri} &$\mathbf{2.71}$ & $\curvearrowswne$ &$\mathbf{69.70}$&$\curvearrowswne$\\
\hline 
\end{tabular}
 \end{table}

 }
\end{example}

\vskip 0.1in
\begin{example}\label{ex:IGA}
{\rm
 In this example, we compare the standard and the extended Krylov approaches again for solving
the standard Lyapunov equation. We consider %equation (\ref{Lyapeq.}) where
 the matrix $A\in\mathbb{R}^{n\times n}$, $n=39304$, coming from the discretization by isogeometric analysis (IGA)
 of the 3D Laplace operator on the unit cube $[0,1]^3$ with zero Dirichlet boundary conditions and a uniform mesh.
Since high degree B-splines are employed as basis functions (here the degree is 4 but higher values are also common),
 this discretization method yields denser stiffness and mass matrices than 
 those typically obtained by low degree 
finite element or finite difference methods; in our experiment, $1.5$\% of the components of $A$ is nonzero.
 See, e.g., \cite{Cottrell2009} for more details on IGA. 

 For the right-hand side we set $C=\mathtt{rand}(n,s)$, $s=3,$ 8, $C=C/\|C\|_F$. 
 In the standard Krylov method the residual norm is computed every 20 iterations.
 The convergence can be checked every $d$ iterations in the extended approach as well, with $d$ moderate
to avoid excessive wasted solves with $A$ at convergence \cite{Simoncini2007}. In our experiments
the computation of the residual norm only takes a small percentage 
of the total execution time and we can afford taking $d=1$.
In both approaches, the residual norm is computed by Algorithm~3.  Table \ref{IGA} collects the results.
\setlength\arrayrulewidth{0.4pt}

 \begin{table}[htb]
%\scriptsize
 \centering
 %\captionsetup{justification=centering,font=scriptsize ,labelfont=bf,textfont=it}
\caption{Example \ref{ex:IGA}. Performance comparison of Standard and Extended Krylov methods. \label{IGA}}
\begin{tabular}{|r|r r r r r r|}
 \hline
 &  $m$ &whole $V_m$& reduced & time res  & two-pass  &time tot  \\
 &   & mem. alloc.& mem. alloc.&  (secs) &  (secs) & (secs) \\
 \hline
 &\multicolumn{6}{|c|}{$s=3$} \\
\hline
 St. Krylov & 280& 840 &9  &1.59 &20.75&44.56 \\
%\hline
 Ex. Krylov & 30 &180& 180 &0.09 & - & 85.54  \\
\hline
 &\multicolumn{6}{|c|}{$s=8$} \\
 \hline
 St. Krylov & 260& 2080 &24 &3.84&45.35&93.49 \\
%\hline
 Ex. Krylov & 27 &216& 216 &0.57 &- & 347.99  \\
 \hline
\end{tabular}
 \end{table}

 The standard Krylov method generates a large space to converge for both values of $s$. 
 Nonetheless, the two-pass strategy allows us to store only 9 basis vectors for $s=3$ and
 24 basis vectors for $s=8$. 
 This feature may be convenient if storage of the whole solution process needs to be allocated in advance.
By checking the residual norm every 20 iterations, the standard Krylov method becomes competitive
with respect to the extended procedure, which is in turn 
penalized by the system solutions with dense coefficient matrices.
Indeed, for $s=3$ the operation
$A^{-1}\mathcal{V}_i^{(2)}$ for $i=1,\ldots,30$ takes 32.75 secs,
that is 38.29\% of the overall
execution time required by the extended Krylov subspace method.
Correspondingly, for $s=8$ the same operation performed during 27 iterations takes 152.92 secs,
that is, 44.94\% of the overall execution time.
This example emphasizes the potential of the enhanced classical approach when system solves are costly,
in which case rational methods pay a higher toll.
 }
 \end{example}

\vskip 0.1in
 \begin{example} \label{SylvLarge}
{\rm
In this example, a Sylvester equation (\ref{Sylveq.}) is solved. The coefficient matrices $A,B\in\mathbb{R}^{n\times n}$,
$n=16384$, come from the discretization by centered finite differences of the 
partial differential operators 
$$ 
{\cal L}_A(u)=(e^{-xy}u_x)_x+(e^{xy}u_y)_y \quad\mbox{and}\quad 
{\cal L}_B(u)=\left(\sin(xy)u_x\right)_x+\left(\cos(xy)u_y\right)_y,
$$
on $[0,1]^2$ with zero Dirichlet boundary conditions.
The right-hand side is a uniformly distributed random matrix where $C_1=\mathtt{rand}(n,s)$, $C_1=C_1/\|C_1\|_F$ and 
$C_2=\mathtt{rand}(n,s)$, $C_2=C_2/\|C_2\|_F$, $s=3$, 8. Since both $A$ and $B$ are large, equation (\ref{Sylveq.})
is solved by the standard Krylov method presented in Section \ref{The Sylvester case} and 
217 iterations are needed to converge for $s=3$, and 145 iterations for $s=8$.
 The residual norm is checked at each iteration and Table~\ref{ex64} collects the results.
 Two approximation spaces, $\mathbf{K}^\square_m(A,C_1)=\mbox{Range}(V_m)$, $\mathbf{K}^\square_m(B,C_2)=\mbox{Range}(U_m)$,
 are generated and a two-pass strategy is employed to cut down the storage demand.
 See Table \ref{ex64.2}.
 %The whole solution process requires $6s=18$ long vectors instead of
%$2sm=1302$, as it would be the case if the whole basis were retained. The second pass
%requires 2.62 seconds, which is a small portion of the total CPU time.

% Table \ref{tab64.1} shows that this procedure drastically reduces the 
%memory requirements of the whole solution process.
\setlength\arrayrulewidth{0.4pt}
\begin{table}[htb]
\begin{center}
\caption{Example \ref{SylvLarge}. CPU times and gain percentages. \label{ex64}}
\begin{tabular}{|r|r c r c|}
\hline
 & time res & gain &time tot & gain \\
 & (secs) &       & (secs)&     \\
 \hline
 &\multicolumn{4}{|c|}{$s=3$ (217 its)}  \\
 \hline
 {\tt lyap} & $60.19$ & $83.6$\%&$65.32$&$76.2$\%  \\
%\hline
%\hline
 {\tt lyap2}  & $74.05$&$86.6$\%& $78.08$&$80.1$\%  \\
% \hline
 {\tt cTri} &{\bf 9.89} &$\curvearrowswne$ &$\mathbf{15.51}$&$\curvearrowswne$\\
\hline 
&\multicolumn{4}{|c|}{$s=8$ (145 its)} \\
\hline
 {\tt lyap} & $201.28$& $88.7$\%&$208.93$&$81.5$\%  \\
%\hline
%\hline
 {\tt lyap2}  & $140.92$&$83.8$\%&$149.95$&$74.2$\% \\
% \hline
 {\tt cTri} &{\bf 22.74} &$\curvearrowswne$ &{\bf 38.65}&$\curvearrowswne$\\
\hline 
\end{tabular}
\end{center}
\end{table}

\setlength\arrayrulewidth{0.4pt}
\begin{table}[ht!]
\begin{center}
 %\captionsetup{justification=centering,font=scriptsize ,labelfont=bf,textfont=it}
\caption{Example \ref{SylvLarge}. Memory requirements with and without full storage, and CPU time of the second Lanczos sweep.
 \label{ex64.2}}
\begin{tabular}{|rrr r r r|} \hline
  &  &  & memory &  reduced  & CPU time  \\
  &  &  & whole $V_m$, $U_m$  & mem. alloc. & (secs) \\
 $n$&$s$ &$m$ & $2s\cdot m$&  $6s$& \\
\hline
16384& 3 & 217& 1032 & 18& 2.62 \\
 16384& 8& 145&2320& 48& 4.93\\
 \hline
\end{tabular}
\end{center}
\end{table}

%\setlength\arrayrulewidth{0.4pt}
%\begin{table}[ht!]
%\begin{center}
% %\captionsetup{justification=centering,font=scriptsize ,labelfont=bf,textfont=it}
%\caption{Example \ref{SylvLarge}. Memory requirements with and without full storage, and CPU time of the second Lanczos sweep.
% \label{tab64.1}}
%\begin{tabular}{|rrr r r r|} \hline
%  &  &  & memory   &   reduced  & CPU time  \\
%  &  &  &  whole $V_m$,$U_m$  & mem. alloc. &  (secs) \\
% $n$&$s$ &$m$ & $2s\cdot m$ &  $6s$& \\
%\hline
%16384 & 3 & 217  & 1302 & 18 & 2.62 \\
% \hline
%\end{tabular}
%\end{center}
%\end{table}
 }
\end{example}

\vskip 0.1in
 \begin{example}\label{ex:Sylvester}
  {\rm
  In this last example, we again consider the Sylvester problem (\ref{Sylveq.}), this time stemming from
 the 3D partial differential equation 
 \begin{equation}\label{operator}
\left(e^{-xy}u_x\right)_x+\left(e^{xy}u_y\right)_y+10u_{zz}=f \quad\mbox{on}\;[0,1]^3,  
 \end{equation}
 %%
 %As in Example \ref{ex:IGA}, the discretization
  %by IGA of the Poisson equation
 %%
%  \begin{equation}\label{operator}
%-\Delta u=f \quad \mbox{on}\; [0,1]^3,  
%  \end{equation}
  %%
  with zero Dirichlet boundary conditions.
  Thanks to the regular domain, 
 its discretization by centered finite differences can be represented by the Sylvester equation
  \begin{equation}\label{SylvEx}
 A X+X B=F, 
  \end{equation}
   where $A\in\mathbb{R}^{n^2\times n^2}$ accounts for the discretization
in the $x, y$ variables, while $B\in\mathbb{R}^{n\times n}$ is associated with the $z$ variable.
   The right-hand side $F\in\mathbb{R}^{n^2\times n}$ takes into account the source term $f$ in agreement
with the space discretization.  %in (\ref{operator}).  {\color{blue} more details about \ref{SylvEx}?}
  See \cite{Palitta2015} for a similar construction. % concerning the convection-diffusion differential operator. 
  %Since $M$ is nonsingular, right multiplication by $M^{-1}$ and left multiplication by 
  %$(M\otimes M)^{-1}=M^{-1}\otimes M^{-1}$ make (\ref{SylvEx}) a standard Sylvester equation.
  
  In our experiment, $n=148$ (so that $n^2=21904$)
 and equation (\ref{SylvEx}) falls into the case addressed in Section~\ref{LargeAsmallB}.
  The right-hand side is $F=-C_1C_2^T$ where $C_1$,$C_2$ are two different normalized random matrices,
  $C_j=\mathtt{rand}(n,s)$, $C_j=C_j/\|C_j\|_F$, $j=1,2$, and $s=3$, 8.
% and $C_2=\mathtt{rand}(n,s)$,
% $C_2=C_2/\|C_2\|_F$, where $s=3$. 
  Convergence is checked at each iteration and Table \ref{tab4} collects the results. 
  \setlength\arrayrulewidth{0.4pt}
 \begin{table}[!ht]
 \centering
  %\captionsetup{justification=centering,font=scriptsize ,labelfont=bf,textfont=it}
\caption{Example \ref{ex:Sylvester}. CPU times and gain percentages. \label{tab4}}
\begin{tabular}{|r| r c r c|}
 \hline
 & time res  & gain &time tot & gain \\
 &  (secs) & gain & (secs)&  \\
 \hline
 &\multicolumn{4}{|c|}{$s=3$ (190 its)} \\
\hline
 {\tt lyap} & $15.47$ &$75.7$\%  &$17.88$ &$63.6$\%  \\
%\hline
 %{\tt lyapchol} & $7.84$ & $32.5$\% &$83.01$ &$5.3$\%  \\
%\hline
 {\tt lyap2}  & $25.35$&$85.2$\%& $27.50$&$76.3$\% \\
% \hline
 {\tt cTri} &$\mathbf{3.76}$ & $\curvearrowswne$ &$\mathbf{6.51}$&$\curvearrowswne$\\
\hline 
&\multicolumn{4}{|c|}{$s=8$ (150 its)} \\
\hline
 {\tt lyap} & $36.99$& $68.2$\%&$40.90$&$60.0$\% \\
%\hline
 {\tt lyap2}  & $77.04$&$84.7$\%& $80.91$&$79.8$\%\\
% \hline
 {\tt cTri} &{\bf 11.77} &$\curvearrowswne$ &{\bf 16.35}&$\curvearrowswne$\\
\hline 
\end{tabular}
 \end{table}
 The method requires $190$ iterations to converge below $10^{-6}$ for $s=3$ and 150 for $s=8$,
and a two-pass strategy allows us to avoid the storage of the whole basis
$V_m\in\mathbb{R}^{n^2\times sm}$. See Table \ref{ex65.2}.

%Therefore the algorithm goes from $ms=558$ to $3s=9$
%memory allocations for the long basis vectors during the whole process. The second
%pass costs only 1 second of CPU time.%
%
\setlength\arrayrulewidth{0.4pt}
\begin{table}[ht!]
\begin{center}
 %\captionsetup{justification=centering,font=scriptsize ,labelfont=bf,textfont=it}
\caption{Example \ref{SylvLarge}. Memory requirements with and without full storage, and CPU time of the second Lanczos sweep.
 \label{ex65.2}}
\begin{tabular}{|rrr r r r|} \hline
  &  &  & memory  &  reduced  & CPU time  \\
  &  &  &  whole $V_m$   & mem. alloc. & (secs) \\
 $n^2$&$s$ &$m$ & $s\cdot m$&  $3s$& \\
\hline
21904& 3 & 190 & 570 & 9&0.93 \\
 21904& 8& 150&1200& 24& 1.31\\
 \hline
\end{tabular}
\end{center}
\end{table}

%Table \ref{tab5} shows how this procedure reduces the memory requirement of the solution
%process at a very low cost.
% \setlength\arrayrulewidth{0.4pt}
%\begin{table}[ht!]
%\begin{center}
% %\captionsetup{justification=centering,font=scriptsize ,labelfont=bf,textfont=it}
%\caption{Example \ref{ex:Sylvester}. Memory requirements with and without full storage, and CPU time of the second Lanczos sweep.
% \label{tab5}}
%\begin{tabular}{|rrr r r r|}
%\hline
%  &  &  & memory   &   reduced  & CPU time  \\
%  &  &  &  whole $V_m$  & mem. alloc. &  (secs) \\
% $n^2$&$s$ &$m$ & $s\cdot m$ &  $3s$& \\
%\hline
%21904 & 3 & 186  & 558 & 9 & 1.00 \\
%% 21904 & 4 & 240 & 960 & 12 & 1.67 \\
% \hline
%\end{tabular}
%\end{center}
%\end{table}
  }
 \end{example}

%%%%%%%%%%%%%%%%%%%%%%%%%%%%%%%%%%%%%%%%%%%%%%%%%%%%%%%%%%%%%%%%%%%%%%%%%%%%%%%%%%%%%%%%%%%%%%%%%%%%%%%%%%%%%%%%%%%%
\section{Conclusions}\label{Conclusions}
We have presented an expression for the residual norm that significantly reduces
the cost of monitoring convergence in
projection methods based on $\mathbf{K}_m^\square$ and $\mathbf{EK}_m^\square$ for Sylvester
and Lyapunov equations and symmetric data.
For the standard Krylov approach, the combination with a two-pass strategy makes this classical
algorithm appealing compared with recently developed methods, 
 both in terms of computational costs and memory requirements, 
whenever data do not allow for cheap system solves. The proposed enhancements 
rely on the symmetric block tridiagonal structure of the projected matrices. In case
this pattern does not arise, as is the case for instance in the nonsymmetric setting,
different approaches must be considered.

\section*{Acknowledgements}
We thank Mattia Tani for providing us with the data of Example \ref{ex:IGA}.
%The proposed procedure is then adapted to Sylvester equations with symmetric coefficient matrices.
\bibliography{%
%/home/valeria/Dropbox/PalittaSimoncini/TESI/bozza_tesi/BIBTEX,%
/home/davide/Scrivania/BIBLIOGRAFIA/BIBTEX}
\end{document}